\documentclass[12pt,reqno]{amsart}
\usepackage[utf8]{inputenc}
\usepackage{amssymb,amsfonts,amsthm}
\usepackage{ulem}
\usepackage{stmaryrd}
\usepackage{amsthm}
\usepackage{amssymb}
\usepackage{mathrsfs}
\usepackage{frcursive}
\usepackage[english,american]{babel}
\usepackage{enumerate}
\usepackage{graphicx}
\usepackage{tikz} 
\usepackage{setspace}
\usepackage{hyperref}
\usepackage{color}

\definecolor{Red}{cmyk}{0,1,1,0.2}

\usepackage{amsfonts,amsmath,layout}

\newcommand{\R}{\mathbb R}

\def\R{\mathbb R}

\newcommand{\be}{\begin{equation}}
\newcommand{\ee}{\end{equation}}

\def\1{{\bf 1}}

\def\ds{\displaystyle}

\newtheorem{Theorem}{Theorem}[section]
\newtheorem{Definition}[Theorem]{Definition}

\newtheorem{Lemma}[Theorem]{Lemma}

\newtheorem{Remark}[Theorem]{Remark}

\begin{document}
\title[strong comparison theorem]{Strong Comparison Principle for Fully Nonlinear Partial differential equations with Hamiltonians Discontinuous in all Variables}
\space
\space
\author[Isaac Ohavi]{Isaac Ohavi}
\address{Kyiv School of Economics, Shpaka Street 3, Kyiv 03113, Ukraine}
\email{iohavi@kse.org.ua}
\dedicatory{Version: \today}
\maketitle
\begin{abstract}
We establish a strong comparison principle for viscosity solutions of fully nonlinear elliptic equations driven by second-order Hamiltonians that may be discontinuous with respect to all variables. 
The analysis relies only on a controlled superlinear growth in the gradient variable together with a structural monotonicity in the unknown, and requires neither ellipticity nor continuity of the coefficients.

The proof is based on the construction of tailored viscosity test functions obtained as solutions of auxiliary eikonal-type equations. These functions compensate for the lack of regularity and allow for a fully local comparison argument despite the complete discontinuity of the Hamiltonian. This yields a robust comparison principle in open subsets of $\mathbb{R}^N$.

As a consequence, we prove the existence of continuous viscosity solutions via a Perron's method adapted to discontinuous frameworks, combined with Ishii’s semicontinuous envelope technique. The class of Hamiltonians covered includes linear and quasilinear equations with merely Borel measurable coefficients, as well as Hamilton–Jacobi–Bellman equations arising in stochastic control and differential games.
\end{abstract}
\section{Introduction}
Let $\Omega$ be an open bounded set of $\R^N$. Given a general fully nonlinear second order Hamiltonian $H$, possibly discontinuous in all its variables
\begin{align*}
H:=\begin{cases}
\Omega\times \R\times \R^N\times \mathbb{S}_N(\R) \to \R,\\
(x,u,p,S)\mapsto H(x,u,p,S),
\end{cases},
\end{align*}
we establish a strong comparison theorem for the following fully nonlinear elliptic problem
\begin{eqnarray}\label{EDP 0}
\begin{cases}
H\big(x,u(x),\nabla u(x),\nabla^2u(x)\big)=0,~~x\in \Omega,\\
u(x)=G(x),~~x\in \partial \Omega,
\end{cases}
\end{eqnarray}
where $G:\partial \Omega \to \R$ denotes a standard Dirichlet boundary condition. The analysis is carried out in the framework of viscosity solutions. The Hamiltonian may be discontinuous with respect to all its variables and is required to satisfy only the following two structural assumptions.\\
- The first assumption is the classical monotonicity condition with respect to the unknown, ensuring a quantified strict increase of the nonlinear term: there exists a strictly positive constant $\lambda>0$, such that
\begin{eqnarray}\label{eq: croiss u}
&\nonumber \forall (x,u,v,p,S)\in \Omega\times \R^2\times \R^N\times \mathbb{S}_N(\R),~~\text{if}~~v\leq u,\\&\text{then}
~~H(x,u,p,S)-H(x,v,p,S)~\ge~\lambda(u-v).
\end{eqnarray}
- The second assumption is a polynomial growth assumption with respect to the gradient variable
\begin{eqnarray}\label{eq: polynom croissance}
\nonumber&\exists m\ge 1,~~\forall (x,u,p,S)\in\Omega\times \R\times \R^N\times \mathbb{S}_N(\R),~~\text{if:}~~|u|\leq M,~~\text{and}~~|S|_{_{\mathbb{M}_N(\R)}}\leq K,\\
&\text{then}~~\exists C_{M,K}>0,~~\text{s.t}:~~|H(x,u,p,S)|\leq C_{M,K}(1+|p|_{_1})^m. 
\end{eqnarray}
Remarkably, the ellipticity condition is not needed in order to establish the comparison theorem. However, in order to remain consistent with the standard viscosity framework, in particular to recover sub and super solution inequalities when viscosity solutions are $C^2$, we assume the standard condition stating that the Hamiltonian $H$ is elliptic, possibly degenerate
\begin{eqnarray*}
&\forall (x,u,p,S,M)\in\Omega\times \R\times \R^N\times \mathbb{S}_N(\R)^2,~~\text{if:}~~S\underset{\mathbb{S}_N(\R)}{\leq} M,\\
&\text{then}~~H(x,u,p,S)\ge H(x,u,p,M).
\end{eqnarray*}
First, we comment on the polynomial growth assumption, since the first assumption is standard in the literature. This requirement is natural and can be illustrated by several classical examples arising in the theory of viscosity solutions for nonlinear PDEs. Endowing the space of symmetric matrices with the sup norm $|\cdot|_{_{\mathbb{S}_N(\R)}}=\underset{1\leq i,j \leq N}{\sup}|(\cdot)_{i,j}|$, we consider the following cases, to which the comparison theorem established in this work applies.
\begin{itemize}
    \item Isaacs–HJB equations with bounded measurable coefficients.\\
    Let $A$ and $B$ be two compact subsets of $\R^N$ and let  
    $b\in L^{\infty}(\Omega\times A\times B,\R^N)$,  
    $a\in L^{\infty}(\Omega\times A\times B,\mathbb{S}_N^+(\R))$,  
    $f\in L^{\infty}(\Omega\times A\times B,\R)$.  
    For all $(x,u,p,S)\in \Omega\times \R\times\R^N\times\mathbb{S}_N(\R)$, consider
    \begin{eqnarray*}
     &H(x,u,p,S)=\\
     &\lambda u+\displaystyle \sup_{a \in A} \inf_{b \in B}~\Big\{~\sum_{i=1}^Np_ib_i(x,a,b)-\sum_{i,j=1}^Na_{i,j}(x,a,b)S_{i,j}+f(x,a,b)~\Big\}.  
    \end{eqnarray*}
    In this case $\lambda>0$, and the constant $C_{M,K}$ appearing in \eqref{eq: polynom croissance} is given by
    $$
    C_{M,K}=\max(\lambda M+|f|_\infty+N^2K|a|_{\infty},|b|_\infty), (m=1).
    $$
\end{itemize}

\begin{itemize}
    \item Note that the above Isaacs–HJB equation simultaneously includes linear elliptic equations with bounded measurable coefficients.
\end{itemize}

\begin{itemize}
  \item Quasilinear equations with bounded measurable coefficients and polynomial growth with respect to the gradient.\\
  Let  
  $b\in L^{\infty}(\Omega\times\R\times \R^N,\R)$,  
  $a\in L^{\infty}(\Omega\times\R\times \R^N,\mathbb{S}_N^+(\R))$,  
  satisfying
  \begin{eqnarray*}
  \exists m\ge 1,~~\exists C_M>0,~\text{s.t}~~\forall (x,u,p,S)\in \Omega\times \R\times\R^N\times\mathbb{S}_N(\R),\\ 
  \text{if}~~|u|\leq M,~~
  \sup_{(x,u,p)}\big\{[|a_{i,j}|,|b|](x,u,p)\}\leq C_M(1+|p|_1)^m,
  \end{eqnarray*}
  then assumption \eqref{eq: polynom croissance} holds for
    \begin{eqnarray*}
     &H(x,u,p,S)=\lambda u+\ds b(x,u,p)-\sum_{i,j=1}^Na_{i,j}(x,u,p)S_{i,j}, 
    \end{eqnarray*}
    with $\lambda>0$, and $C_{M,K}=\lambda M+C_M+N^2KC_M$.
\end{itemize}

\begin{itemize}
    \item Monge–Ampère equation.\\ 
    In the context of the Monge–Ampère equation in its well-posed formulation
    \begin{eqnarray*}
     \det(\nabla^2u(x))=f(x,u(x),\nabla u(x)),~~x\in \Omega,   
    \end{eqnarray*}
    namely assuming that $u$ is convex and that $f$ is a positive measurable function satisfying assumptions \eqref{eq: croiss u}–\eqref{eq: polynom croissance}, we clearly obtain
    $$
    C_{M,K}=C_M+N!K^N.
    $$
\end{itemize}
The main contribution of this work is to introduce a fundamentally different approach to comparison principles, which completely bypasses both the doubling-of-variables technique and Ishii's lemma. In particular, our method does not require any ellipticity assumption, nor any form of continuity of the nonlinear term, and applies to Hamiltonians that may be arbitrarily discontinuous in all variables. \\
We emphasize that this is not unexpected, since even within the classical doubling-of-variables framework, it is easy to verify that for first-order equations of the form $\lambda u + H(p)$, with an arbitrary discontinuous function $H$, the comparison principle still holds. In contrast, the situation for existence is fundamentally different: minimal measurability or semicontinuity assumptions are typically required to construct solutions.\\
We also establish the existence of continuous viscosity solutions under assumptions \eqref{eq: croiss u}–\eqref{eq: polynom croissance}, using Ishii’s notion of viscosity solutions based on semicontinuous envelopes, suitably adapted to the present framework.\\
A notable feature of our approach is that the Hamiltonian $H$ is used directly in its original form in the comparison theorem, without resorting to the upper and lower semicontinuous envelopes appearing in Ishii’s definition of viscosity solutions. Our goal is to preserve, as much as possible, the original structure of the Hamiltonian in the definition of viscosity solutions, thereby retaining the maximal amount of information. It should also be emphasized that stability issues are not addressed in the present work, since for discontinuous Hamiltonians there is, in general, no reasonable hope of obtaining such results.

Given a super solution $u$ (respectively, a sub solution $v$), rather than comparing the equation at two nearby points, our strategy is entirely local and relies on the construction of explicit test functions at a single point where the maximum of $v-u$ is achieved, leading to a contradiction. Importantly, no strict positivity assumption is imposed on this maximum, which further distinguishes our approach from the classical framework.\\
The core of the method consists in constructing test functions solving suitable eikonal-type equations in a neighborhood $\mathcal{O}$ of the point where the supremum is attained. These functions are designed to satisfy strict viscosity inequalities throughout the neighborhood, while retaining precise control of their second-order behavior. This allows us to encode the structure of the equation directly into the test functions, thereby bypassing any requirement of regularity on the Hamiltonian.\\
A key new ingredient is a localization mechanism based on carefully chosen boundary conditions on a small boundary $\partial \mathcal{O}$ of the neighborhood of the supremum point, which forces the extrema of the perturbed functions to occur in prescribed regions. 
Combining these constructions with a refined comparison argument based on explicit exponential estimates, we derive a contradiction without invoking any continuity or ellipticity properties.\\
As a consequence, we obtain the sharp boundary characterization
$$
\sup_{\overline{\Omega}} (v-u) = \sup_{\partial \Omega} (v-u),
$$
which strengthens the classical comparison principle and justifies the terminology
\[
\textbf{Strong Comparison Principle}
\]
used throughout this work.\\
Beyond its immediate consequences, this approach opens a new direction in viscosity theory, allowing one to treat equations with extremely low regularity and providing a flexible framework that is stable under arbitrary discontinuities.\\
From this perspective, our result can be interpreted as capturing an effective macroscopic behavior emerging from highly irregular microscopic structures. In particular, even in the presence of arbitrary discontinuities, the comparison principle enforces a unique global behavior and yields the existence and uniqueness of a continuous viscosity solution.\\

The comparison principle lies at the heart of the theory of nonlinear partial differential equations, see \cite{CrandallIshiiLions1992}. 
In the context of Hamilton--Jacobi and fully nonlinear second-order equations, it provides the fundamental mechanism ensuring uniqueness, stability and well-posedness, and forms the analytical backbone of stochastic control, differential games and martingale problems.\\
Since the seminal developments of viscosity solution theory, comparison results have traditionally relied on a powerful but highly nontrivial analytical framework built upon the doubling-of-variables method, Ishii's lemma and delicate stability arguments involving semicontinuous envelopes and compactness. 
This machinery has proved remarkably robust and has led to a vast and successful theory. However, it is also intrinsically tied to structural assumptions such as continuity of the Hamiltonian and ellipticity. When these assumptions fail, the classical approach encounters severe limitations.\\
A major open difficulty in the field is the treatment of equations whose coefficients are merely Borel measurable and whose Hamiltonians may exhibit arbitrary discontinuities in all variables. In such regimes, the traditional viscosity toolbox becomes fragile: doubling-of-variables arguments introduce penalization parameters whose limiting behaviour is incompatible with discontinuities, while stability techniques rely on compactness properties that are no longer available. 
These obstacles have long prevented the emergence of a general comparison theory in the fully discontinuous setting, even for linear equations.

The main contribution of this work is that comparison for fully nonlinear equations does not intrinsically require the classical viscosity machinery. Instead, it can be obtained through the construction of a single global barrier solving an auxiliary eikonal equation.\\
We believe that this new viewpoint opens the way to a systematic treatment of nonlinear PDEs in discontinuous environments and provides a new foundation for the analysis of stochastic control and differential games with irregular data.
Beyond the framework of continuous viscosity solutions, there also exists a vast theory of $L^p$ viscosity solutions. However, the comparison is mainly often stated between
$L^p$ viscosity solution with a strong $W^{2,p}$
solution rather than directly comparing two viscosity solutions. This is mainly due to the fact that the $L^p$ framework lacks pointwise second-order information. In this setting, the Hessian exists only almost everywhere, so the classical viscosity machinery based on jets and pointwise test functions cannot be applied directly. Strong solutions provide enough regularity to interpret the equation almost everywhere and allow the use of Aleksandrov--Bakelman--Pucci (ABP) estimates and Sobolev techniques. The comparison between two $L^p$ viscosity solutions is then recovered indirectly through approximation and stability arguments.

We now turn to a discussion of related work concerning comparison principles for discontinuous Hamiltonians. An exhaustive account is beyond the scope of the present work, in view of the breadth and diversity of the literature, and would require a much more extensive treatment. For the sake of clarity, we adopt a structured presentation and review the main results involving Borel measurable coefficients, proceeding from the linear setting to nonlinear elliptic equations. We deliberately focus on operators in non-divergence form. Indeed, equations in divergence form can often be handled through variational methods in appropriate Sobolev spaces, even for merely measurable coefficients, whereas the non-divergence framework lacks such a structure and remains significantly more challenging.

\subsection{Linear operators and quasi linear equations with borel coefficients}
Let us begin with the classical counterexample of Nadirashvili concerning nonuniqueness for linear elliptic operators with merely measurable coefficients. Consider
\[
L=\sum_{i,j} a_{ij} D_{ij}
\]
in the unit ball $B_1\subset\mathbb{R}^N$. Nadirashvili \cite{Nadirashvili1997} showed that weak solutions of the associated Dirichlet problem are not unique in dimension $N\ge3$. More precisely, he constructed two sequences of smooth uniformly elliptic matrices $(a^{0,k}_{ij})$ and $(a^{1,k}_{ij})$ converging almost everywhere to the same limit $a_{ij}$, while the corresponding classical solutions converge to different limits. Thus, the limiting equation admits multiple weak solutions.

Safonov later refined this construction by allowing
\[
|a^{0,k}_{ij}-a^{1,k}_{ij}|\le \Lambda,
\]
for arbitrary $\Lambda>0$, while still producing distinct weak limits \cite{SafonovNonuniq}. In contrast, for continuous coefficients, the Dirichlet problem admits a unique strong solution in $W^{2,p}_{\mathrm{loc}}$ for every $1<p<\infty$ \cite[Sec.~9.6]{GilbargTrudinger1983}.

For merely measurable coefficients, the notion of solution depends on the chosen approximation procedure. A standard approach is to approximate the coefficients by smooth ones and study limits of classical solutions. By the Krylov--Safonov theory \cite{GilbargTrudinger1983,KrylovSafonov1980,KrylovSafonov1981}, these solutions are uniformly H\"older continuous and subsequences converge uniformly to continuous limits. Jensen \cite{Jensen1988} showed that viscosity solutions are stable under almost everywhere convergence of the coefficients, so that these limits coincide with viscosity solutions of the limiting equation. Our main result shows that the addition of a strictly monotone zeroth-order term $\lambda u$, $\lambda>0$, restores comparison and uniqueness at the viscosity level, even under arbitrary discontinuities.

A central tool in the classical theory is the ABP estimate
\[
\sup_{\Omega}|u|
\le \sup_{\partial\Omega}|u| + C\,\|Lu\|_{L^{N}(\Omega)},
\]
which implies comparison and uniqueness under suitable $W^{2,N}$ control. Together with Calderón--Zygmund theory \cite{CalderonZygmund1952} and Sobolev estimates \cite{GilbargTrudinger1983}, it yields well-posedness for uniformly elliptic equations. The estimate is critical with respect to the exponent $N$, and in general $W^{2,p}$ estimates hold only for $p>N$, while both regularity and uniqueness may fail for $p<N$.

Beyond the Sobolev framework, uniqueness results have been obtained under progressively weaker regularity assumptions on the coefficients. Caffarelli established uniqueness when the coefficients are discontinuous at a single point \cite{Caffarelli1989}. This was extended by Cerutti--Escauriaza--Fabes to more general singular sets \cite{CeruttiEscauriazaFabes1,CeruttiEscauriazaFabes2}, and by Krylov to sets with countable closure \cite{KrylovWeakUniq}. Safonov further extended the theory to closed sets of small Hausdorff dimension \cite{SafonovNonuniq}, with additional contributions in
\cite{BassPardoux,EscauriazaBounds,FabesStroock,FabesKenig}. These results highlight the delicate interplay between ellipticity, regularity, and uniqueness in low-regularity regimes.

Weak uniqueness is also closely related to the well-posedness of martingale problems and diffusion processes \cite{StroockVaradhan1979}. While uniqueness holds in the continuous setting and in low dimensions, the high-dimensional measurable case remains substantially more delicate. In degenerate settings, viscosity solutions may still select a distinguished evolution, yielding an associated Markov process; see \cite{Kumar,Krylov1973}.

A major development in low-regularity elliptic theory is the vanishing mean oscillation (VMO) framework. The monograph \cite{Maugeri1999} develops an $L^p$ theory for equations with coefficients in VMO, introduced by Sarason \cite{Sarason1975}. This extends Calderón--Zygmund theory beyond continuity and yields existence and uniqueness in Sobolev spaces under structural assumptions.

For uniformly elliptic operators with VMO coefficients, Chiarenza--Frasca--Longo
\cite{chiarenza1991, chiarenza1993} established local $W^{2,p}$ estimates for all $p\in(1,\infty)$ via perturbative arguments. Krylov further developed a systematic $W^{2,p}$ theory in this setting \cite{DongKrylov2010VMO,Krylov1987, Krylov1996lectures,KrylovVMO2007}, including rough lower-order terms and VMO principal coefficients.

More recent works address borderline regimes. Krylov \cite{Krylov2023Morrey} studied operators with almost-VMO coefficients and Morrey-class lower-order terms, obtaining $W^{2,p}$ well-posedness for $p>N$. Related results include those of Lee \cite{Lee2025Div,Lee2025General} for low-integrability drifts, and Chernobai--Shilkin \cite{ChernobaiShilkin2024} for singular Morrey structures.

These linear theories provide the analytical backbone for nonlinear problems. Existence results for quasilinear equations are typically based on Bernstein-type estimates combined with continuation or compactness arguments; see \cite{GilbargTrudinger1983,Lady1968,Maugeri1999}. The $L^p$ linear theory supplies the stability and compactness tools required to implement these nonlinear schemes in rough-coefficient settings.
\subsection{Fully non linear framework}
We first refer to the recent monograph \cite{Barlesbook}, which surveys several modern developments on Hamilton--Jacobi equations and optimal control problems with discontinuities in the first-order setting, together with the main analytical difficulties arising in such frameworks.

A natural first step in the study of discontinuous Hamiltonians consists in considering singularities localized at a single point. Star-shaped networks provide a canonical framework for this purpose, already in the simplest configuration of two edges coupled through Dirichlet or Neumann-type transmission conditions. These models constitute a fundamental testing ground before addressing more general discontinuous structures.

A substantial literature has been devoted to first- and second-order equations posed on networks, often involving vanishing viscosity effects at the junction. For quasilinear and semilinear equations on star-shaped networks, where solutions are assumed to be differentiable at the junction point, we refer for instance to
\cite{ Achdou 1, Achdou 2, Barles semi linear, Camilli 1, Camilli 2, OhaviPDE, Below 3}.
From the viewpoint of deterministic optimal control, Hamilton--Jacobi equations with convex Hamiltonians on networks have been extensively studied; see, e.g.,
\cite{control 1, control 2, control 3, control 4, control 5, control 6}.
We also mention recent developments concerning systems of conservation laws on junctions
\cite{Carda junction -1, Carda junction 0, Carda junction 1, Carda junction 2}.

Major advances in this direction are due to Lions and Souganidis.
In \cite{Lions Souganidis 1}, they introduced a notion of state-constraint viscosity solutions for one-dimensional junction problems associated with first-order Hamilton--Jacobi equations with possibly nonconvex coercive Hamiltonians, and established well-posedness and stability results.
This analysis was later extended in \cite{Lions Souganidis 2} to multidimensional junctions for first- and second-order equations with Kirchhoff-type Neumann conditions.

A key feature of this framework is that uniqueness is obtained under the assumption that solutions are Lipschitz continuous at the junction point, thereby avoiding the doubling-of-variables method and substantially simplifying the comparison argument.
Moreover, flux-limited solutions in the sense of Imbert and Monneau
\cite{control 5, control 6}
are shown to be equivalent to generalized Kirchhoff solutions for a suitable choice of flux limiter.

The approach developed in the present work is closely related to the breakthrough contributions
\cite{Ohavi visco 1, Ohavi visco 2, Ohavi control},
where a new method based on the construction of explicit viscosity test functions on one-dimensional star-shaped networks was introduced.
One of the main insights of these works is that generalized Kirchhoff viscosity solutions in the sense of Lions--Souganidis coincide with classical Kirchhoff-type viscosity solutions for second-order, possibly degenerate, equations, without requiring any explicit value of the Hamiltonian at the junction point.
From this viewpoint, the junction behaves as a transmission interface rather than as a boundary condition, while the coupling between the Hamiltonians is encoded implicitly through the Kirchhoff condition.

This interpretation is closely connected with the flux-limited framework of Imbert and Monneau \cite{control 5}, where the flux limiter may be interpreted as prescribing an effective value of the solution at the junction, thereby inducing an equivalent Kirchhoff condition.

Finally, \cite{Ohavi visco 2, Ohavi control} introduced a more sophisticated transmission condition, referred to as a {\it nonlinear local-time Kirchhoff condition}. This condition admits a natural probabilistic interpretation in terms of stochastic control and scattering phenomena for spider diffusion processes, where the local time variable measures the amount of time spent by the diffusion at the vertex. This leads to a new class of control problems involving dynamics depending explicitly on the local time variable \cite{Ohavi control}.
\\
We also mention the theory of Cordes-type conditions for nonlinear elliptic operators with discontinuous coefficients developed in \cite[Chapter~3]{Maugeri1999}. In this framework, nonlinear equations with measurable coefficients are treated under structural assumptions ensuring that the nonlinear operator remains sufficiently close, in a quantitative sense, to a uniformly elliptic operator with regular coefficients. Combined with perturbative Calderón--Zygmund techniques, these conditions yield existence, uniqueness, and $W^{2,p}$ regularity results for strong solutions, thereby extending the classical linear theory to a broad class of nonlinear equations with discontinuous data.

A fundamental step in the development of the theory of fully nonlinear elliptic equations with low regularity data is due to Caffarelli. In \cite{Caffarelli1989Annals}, he established interior $W^{2,p}$ and Hölder-type a priori estimates for fully nonlinear uniformly elliptic equations under structural assumptions, providing one of the first robust regularity frameworks beyond linear theory. Later, in collaboration with Crandall, Kocan and Swiech \cite{CaffarelliCrandallKocanSwiech1996}, an $L^p$ viscosity solution theory was developed for fully nonlinear equations with merely measurable coefficients, showing that well-posedness and stability can be recovered even in the absence of continuity of the data. These results play a central role in the modern theory of viscosity solutions and constitute a natural precursor to comparison principles in discontinuous or low-regularity settings, which is the main focus of the present work.\\
The study of discontinuous Hamilton--Jacobi equations was initiated by Ishii in his seminal work \cite{Ishii1985}, where Hamiltonians measurable in time and possibly discontinuous in the remaining variables were considered. A central difficulty in this framework arises from the special role of the time derivative, which necessitates a specific treatment of discontinuities in the $t$-variable.

In the simplest setting of ODEs, Ishii’s key idea can be understood by decomposing the Hamiltonian and by introducing an appropriate integral transformation that absorbs the irregular part. This yields an equivalent formulation in which viscosity solutions are defined for a regularized equation, while preserving continuity of the unknown.

This viewpoint was further clarified by Lions and Perthame \cite{LionsPerthame1987}, who established equivalent formulations of Ishii’s notion of solution, and later extended to second-order equations by Nunziante \cite{Nunziante1990,Nunziante1991}. These developments supported the general principle that many results valid for continuous Hamiltonians should remain true under mere measurability in time, although their proofs typically require nontrivial technical adaptations.

The paper is organized as follows. In Section \ref{sec inro}, we introduce the main notations and state the principal results of this work. Section \ref{sec example} provides a simple proof of the comparison theorem in the one-dimensional first-order setting with discontinuous Hamiltonians, in order to give the reader an initial intuition of the main ideas underlying the method. Section \ref{sec preuve} is devoted to the proof of the main result, namely the strong comparison principle for fully nonlinear equations with discontinuous Hamiltonians. Finally, in Section \ref{sec existence}, we discuss well-posedness issues and establish existence results using Ishii’s framework based on semicontinuous envelopes.

\section{Notations, Assumptions and Definitions}\label{sec inro}
In this Section, we introduce the main mathematical background, definitions and the main result of this contribution.
In all this work, $\Omega$ is open bounded subset of $\R^N$ ($N\ge1$) and $\partial \Omega$ refers to its boundary. The notations $(|\cdot|_{_1},|\cdot|_{_{_2}})$ refer to the following usual norms defined $\forall x=(x_1,\cdots,x_N)\in \R^N$ by
$$|x|_{_{1}}=\sum_{i=1}^N|x_i|,~~|x|_{_{_2}}=\Big(\sum_{i=1}^N|x_i|^2\Big)^{1/2}.$$
Given $x_0\in \R^N$ and $\varepsilon>0$, we will denote by $\mathcal{B}(x_0,\varepsilon)$ the ball of center $x_0$ and radius $\varepsilon$
$$\mathcal{B}(x_0,\varepsilon):=\Big\{~x\in \R^N,~~|x-x_0|_{_{_2}}\leq \varepsilon~\Big\},$$
whereas $\partial \mathcal{B}(x_0,\varepsilon)$ refers to its boundary.
The space of square real matrices of size $N$ is denoted  $\mathbb{M}_N(\R)$, which we endow with the following norm defined $\forall A=(a_{i,j})_{1\leq i,j \leq N}\in \mathbb{M}_N(\R)$ by
$$|A|_{_{\mathbb{M}_N(\R)}}=\sup_{1\leq i,j \leq N}|a_{i,j}|.$$
We denote also by $\mathbb{S}_N(\R)$ the sub space of $\mathbb{M}_N(\R)$ consisting on the set of symmetric matrices.\\
Now we introduce the following fully nonlinear Hamiltonian
\begin{align*}
H:=\begin{cases}
\Omega\times R\times \R^N\times \mathbb{S}_N(\R) \to \R,\\ 
(x,u,p,S)\mapsto H(x,u,p,S),
\end{cases}
\end{align*}
We assume that the Hamiltonian $H$ satisfies the following assumptions
$$\textbf{Assumption}~(\mathcal{H})$$
a) The following standard growth assumption for the nonlinear term $H$ with respect to their second variable: there exists a strictly positive constant $\lambda>0$, such that
\begin{eqnarray*}
& \forall (x,u,v,p,S)\in \Omega\times \R^2\times \R^N\times \mathbb{S}_N(\R),~~\text{if}~~v\leq u,~\text{then}\\
&H(x,u,p,S)-H(x,v,p,S)~\ge~\lambda(u-v).
\end{eqnarray*}
b) A polynomial growth condition with respect to the gradient of the nonlinear terms $H$
\begin{eqnarray*}
&\exists m\ge 1,~~\forall (x,u,p,S)\in\Omega\times \R\times \R^N\times \mathbb{S}_N(\R),~~\text{if:}~~|u|\leq M,~~\text{and}~~|S|_{_{\mathbb{M}_N(\R)}}\leq K,\\
&\text{then}~~\exists C_{M,K}>0,~~\text{s.t}:~~|H(x,u,p,S)|\leq C_{M,K}(1+|p|_{_1})^m. \end{eqnarray*}
c) The Hamiltonian $H$ is elliptic, possibly degenerate in the sense that
\begin{eqnarray*}
&\forall (x,u,p,S,M)\in\Omega\times \R\times \R^N\times \mathbb{S}_N(\R)^2,~~\text{if:}~~S\underset{\mathbb{S}_N(\R)}{\leq} M,\\
&\text{then}~~H(x,u,p,S)\ge H(x,u,p,M).\end{eqnarray*}
As explained in the Introduction, assumption $(\mathcal{H}),~-c)$ is not used in the proof of the strong comparison principle. It is included only to remain compatible with the standard viscosity framework in the case of regular solutions.
In all this work, given $G:\partial \Omega \to \R$, our concern will be to obtain a strong comparison theorem for super/sub viscosity discontinuous solutions of the following fully nonlinear Dirichlet problem
\begin{eqnarray}\label{eq: EDP 1}
\begin{cases}
 H\big(x,u(x),\nabla u(x),\nabla^2u(x)\big)=0,~~x\in \Omega,\\
 u(x)=G(x),~~x\in \partial \Omega,
\end{cases}
\end{eqnarray}
only under assumption $(\mathcal{H})$, whereas the Hamiltonian is allowed to be discontinuous in all its variables.\\
We next recall the notion of viscosity solution used throughout this work. Contrary to Ishii's framework, we do not introduce upper and lower semicontinuous envelopes of the Hamiltonian itself.\\
Let
$$
u:\overline{\Omega}\to \R,
\qquad
\Big(\text{resp. }v:\overline{\Omega}\to \R\Big),
$$
be locally bounded functions. We denote by $u_\star$ the lower semicontinuous envelope of $u$, and by $v^\star$ the upper semicontinuous envelope of $v$, namely
$$
\forall x\in\overline{\Omega},
\qquad
u_\star(x)=\liminf_{y\to x}u(y),
\qquad
\Big(\text{resp. }v^\star(x)=\limsup_{y\to x}v(y)\Big).
$$

We now introduce the definition of viscosity solutions adapted to discontinuous Hamiltonians in the present framework. A key point is that super- and sub-solutions are assumed to be locally bounded, rather than merely locally lower or upper bounded. This additional assumption is essential for the construction of the test functions used in the proof of the comparison theorem.
\begin{Definition}\label{def: solu de visco}

a) A locally bounded map $u:\overline{\Omega} \to \R$ is called viscosity super solution of \eqref{eq: EDP 1}, if for all test function $\phi \in \mathcal{C}^{2}\big(\overline{\Omega},\R)$, and for all local minimum point $x\in \Omega$ of $u_{\star}-\overline{\phi}$, with $u_{\star}(x)-\phi(x)=0$, one has
$$H\big(x,u_\star(x),\nabla \phi(x),\nabla^2\phi(x)\big)\ge 0.$$
b) A locally bounded map $v:\overline{\Omega} \to \R$ is called viscosity sub solution of \eqref{eq: EDP 1}, if for all test function $\phi \in \mathcal{C}^{2}\big(\overline{\Omega},\R)$, and for all local maximum point $x\in \Omega$ of $v^{\star}-\overline{\phi}$, with $v^{\star}(x)-\phi(x)=0$, one has
$$H\big(x,v^\star(x),\nabla \phi(x),\nabla^2\phi(x)\big)\leq 0.$$
c) We say that a locally bounded map $u:\overline{\Omega} \to \R$ is a viscosity solution of \eqref{eq: EDP 1}, if it is both a viscosity super and sub solution  of \eqref{eq: EDP 1}.
\end{Definition}
The main result of this work is the following Strong Comparison Principle:
\begin{Theorem} (Strong Comparison Principle.)\label{th : comparison}
Assume assumption $(\mathcal{H})$. Let $u$ (resp. $v$) a viscosity super solution (resp. sub solution) of \eqref{eq: EDP 1}. Then we have
\begin{equation}\label{eq: princip maximum}
 \sup\Big\{~v^{\star}(x)-u_{\star}(x),~x\in \overline{\Omega}~\Big\}=\sup\Big\{~v^{\star}(x)-u_{\star}(x),~x\in \partial\Omega~\Big\}.  
\end{equation}
\end{Theorem}
One of the key ingredients in the proof of the strong comparison principle is the following lemma.
\begin{Lemma}\label{lm of strong comparison theorem}
Let $x_0\in \Omega$ and $\varepsilon>0$ such that $\mathcal{B}(x_0,\varepsilon)\subset \Omega$. We consider 
$$u:=\overline{\mathcal{B}(x_0,\varepsilon)}\to \R,~~v:=\overline{\mathcal{B}(x_0,\varepsilon)}\to \R,$$
two maps satisfying for some $\underline{x}_\varepsilon  \in\overline{\mathcal{B}(x_0,\varepsilon)},~\underline{x}_\varepsilon \neq x_0$
$$v(x_0)-u(x_0)\ge v(\underline{x}_\varepsilon )-u(\underline{x}_\varepsilon).$$
Fix $\theta>0$, and define by $A_\varepsilon^u$ and $B_\varepsilon^v$ the following constants
\begin{eqnarray*}
&A_\varepsilon^u=
\begin{cases}\ds \frac{u(x_0)-u(\underline{x}_\varepsilon)}{1-\exp( \theta\varepsilon)}\big[\exp(-\theta\varepsilon)-1\big],~&\text{if}~u(x_0)-u(\underline{x}_\varepsilon) \ge 0,\\
\ds \frac{u(x_0)-u(\underline{x}_\varepsilon)}{1-\exp(-\theta\varepsilon)}\big[\exp(\ds \theta\varepsilon)-1\big],~&\text{if}~u(x_0)-u(\underline{x}_\varepsilon)\leq 0.
\end{cases}
\\
&B_\varepsilon^v=
\begin{cases}\ds \frac{v(x_0)-v(\underline{x}_\varepsilon)}{1-\exp(-\theta\varepsilon)}\big[\exp(\ds \theta\varepsilon)-1\big],~&\text{if}~v(x_0)-v(\underline{x}_\varepsilon) \ge 0,\\
\ds \frac{v(x_0)-v(\underline{x}_\varepsilon)}{1-\exp(\theta\varepsilon)}\big[\exp(-\theta\varepsilon)-1\big],~&\text{if}~v(x_0)-v(\underline{x}_\varepsilon)\leq 0.
\end{cases}
\end{eqnarray*}
Then
$$B_\varepsilon^v\ge A_\varepsilon^u$$
\end{Lemma}
\begin{proof}
Recall that by assumption we have
$$v(x_0)-v(\underline{x}_\varepsilon)\ge u(x_0)-u(\underline{x}_\varepsilon).$$
Hence we need to treat only the following three cases, which are
\begin{align*}
&a)~~v(x_0)-v(\underline{x}_\varepsilon)\ge 0,~~u(x_0)-u(\underline{x}_\varepsilon)\ge 0,\\
&b)~~v(x_0)-v(\underline{x}_\varepsilon)\ge 0,~~u(x_0)-u(\underline{x}_\varepsilon)\leq 0,\\
&c)~~v(x_0)-v(\underline{x}_\varepsilon)\leq 0,~~u(x_0)-u(\underline{x}_\varepsilon)\leq 0.
\end{align*}
Recall first that for all $\theta \ge 0,~\varepsilon\ge 0$: $(1-\exp(-\theta\varepsilon))^2\leq (\exp(\theta\varepsilon)-1)^2$, implying that the following constant 
\begin{eqnarray}\label{eq calcul 1}
\rho(\theta,\varepsilon)=\frac{1-\exp(-\theta\varepsilon)}{\exp(\theta\varepsilon)-1}+\frac{1-\exp(\theta\varepsilon)}{1-\exp(-\theta\varepsilon)}=\frac{(1-\exp(-\theta\varepsilon))^2-(\exp(\theta\varepsilon)-1)^2}{(\exp(\theta\varepsilon)-1)(1-\exp(-\theta\varepsilon))}\leq 0, 
\end{eqnarray}
is negative.
For the first case a) we have
\begin{align*}
&A_\varepsilon^u-B_\varepsilon^v=\\
&\ds\frac{\Big[u(x_0)-u(\underline{x}_\varepsilon)\Big]}{\exp(\theta\varepsilon)-1}[1-\exp(-\theta \varepsilon)]+\frac{\Big[-v(x_0)+v(\underline{x}_\varepsilon)\Big]}{1-\exp(-\theta\varepsilon)}[\exp(\theta\varepsilon)-1]\\
&\leq [v(x_0)-v(\underline{x}_\varepsilon)]\ds \Big(\frac{1-\exp(-\theta\varepsilon)}{\exp(\theta\varepsilon)-1}+\frac{1-\exp(\theta\varepsilon)}{1-\exp(-\theta\varepsilon)}\Big)\\
&\leq [v(x_0)-v(\underline{x}_\varepsilon)]\rho(\theta,\varepsilon),
\end{align*}
using that $u(x_0)-u(\underline{x}_\varepsilon)\leq v(x_0)-v(\underline{x}_\varepsilon)$. Hence from $v(x_0)-v(\underline{x}_\varepsilon)\ge 0$, it follows that  
$$A_\varepsilon^u-B_\varepsilon^v\leq 0.$$
For the case b), calculations lead to
\begin{eqnarray*}
&A_\varepsilon^u-B_\varepsilon^v= \ds\frac{\Big[-v(x_0)+v(\underline{x}_\varepsilon)+u(x_0)-u(\underline{x}_\varepsilon)\Big]}{1-\exp(-\theta\varepsilon)}\big[\exp(\theta \varepsilon)-1\big]\leq 0,
\end{eqnarray*} 
using once again $u(x_0)-u(\underline{x}_\varepsilon)\leq v(x_0)-v(\underline{x}_\varepsilon)$. 
Finally for the last case c), we have
\begin{align*}
&A_\varepsilon^u-B_\varepsilon^v=\\
&=\ds \frac{\Big[u(x_0)-u(\underline{x}_\varepsilon)\Big]}{1-\exp(-\theta\varepsilon)}[\exp(\theta\varepsilon)-1]+\frac{\Big[-v(x_0)+v(\underline{x}_\varepsilon)\Big]}{\exp(\theta\varepsilon)-1}[1-\exp(-\theta\varepsilon)]\\
&\leq[v(x_0)-v(\underline{x}_\varepsilon)]\ds \Big(-\frac{1-\exp(-\theta\varepsilon)}{\exp(\theta\varepsilon)-1}+\frac{\exp(\theta\varepsilon)-1}{1-\exp(-\theta\varepsilon)}\Big)\\
&\leq-[v(x_0)-v(\underline{x}_\varepsilon)]\rho(\theta,\varepsilon)\leq 0, 
\end{align*}
using once again \eqref{eq calcul 1} and that $v(x_0)-v(\underline{x}_\varepsilon)\leq 0$. The proof is complete.
\end{proof}

\section{An Example of introduction for first order problems in $1$-D}\label{sec example} 
In this subsection, we prove a strong comparison theorem for discontinuous Hamiltonians in the first order case, posed on a compact set $[0,a]~(a>0)$ of $\R$ and under the Lipschitz growth assumption $(\mathcal{H}),~-b)$. The purpose of this one-dimensional setting is to illustrate the main ideas of the proof in a simpler framework before extending them to second-order equations in higher dimensions.
Assume then assumption $(\mathcal{H})$, where $b)$ is naturally replaced by
\begin{eqnarray}\label{hyp 1 order}
\forall (x,u,p)\in\Omega\times \R^2,~\text{if:}~|u|\leq M,~\text{then}~~\exists C_{M}>0,~\text{s.t}:~|H(x,u,p)|\leq C_{M}(1+|p|), \end{eqnarray}
(here we have assumed $m=1$ to simplify the calculations).\\ 
We consider then the fully non linear first order problem posed on the compact set $(0,a)$
\begin{eqnarray}\label{PDE 1 ordre}
H(x,\mathfrak{w}(x),\partial_x\mathfrak{w}(x))=0,~~x\in (0,a),~~\mathfrak{w}(0)=e_0,~~\mathfrak{w}(a)=e_a,
\end{eqnarray}
with boundary conditions $(e_0,e_a)\in \R^2$. Let $u:[0,a]\to \R$ (resp. $v:[0,a]\to \R$) a locally bounded super solution (resp. locally bounded sub solution) of \eqref{PDE 1 ordre}.
We assume by contradiction that the following supremum $\sup_{x\in[0,a]}\big\{v^\star(x)-u_\star(x)\big\}$ is reached at some $x_0\in (0,a)$, namely
\begin{eqnarray}\label{eq contra ordre 1}
 \sup_{x\in[0,a]}\big\{~v^\star(x)-u_\star(x)~\big\}=v^\star(x_0)-u_\star(x_0),~~x_0\in (0,a).   
\end{eqnarray}
Fix $\varepsilon>0$ small enough and consider the following neighborhood $(x_0-\varepsilon,x_0+\varepsilon) \subset (0,a)$ of $x_0$. By assumption, there exists $M_{\varepsilon,x_0}=M$ such that
\begin{eqnarray}\label{born u v}\forall x\in [x_0-\varepsilon,x_0+\varepsilon],~~|u_\star(x)|\leq M,~~|v^\star(x)|\leq M.\end{eqnarray}
Let $\eta>0$ be a parameter that will be fixed later in the proof (see inequation \eqref{eq eta fixation}). We introduce the following maps defined on $[x_0-\varepsilon,x_0+\varepsilon]$, s.t. $\forall x \in [x_0-\varepsilon,x_0+\varepsilon]$
\begin{eqnarray}\label{expression function test 1}
&\overline{\phi}(x)=
\begin{cases}\ds \frac{u_\star(x_0)-u_\star(x_0-\varepsilon)}{1-\exp(\frac{\lambda}{C}\varepsilon)}\exp(-\frac{\lambda}{C}(x-x_0))-\frac{C+\eta}{\lambda}+M,\\
\text{if}~u_\star(x_0)-u_\star(x_0-\varepsilon) \ge 0,\\
\ds \frac{u_\star(x_0)-u_\star(x_0-\varepsilon)}{1-\exp(-\frac{\lambda}{C}\varepsilon)}\exp(\frac{\lambda}{C}(x-x_0))-\frac{C+\eta}{\lambda}+M,\\
\text{if}~u_\star(x_0)-u_\star(x_0-\varepsilon)\leq 0,
\end{cases},\\
\label{expression function test 2}
&\underline{\phi}(x)=
\begin{cases}\ds \frac{v^\star(x_0)-v^\star(x_0-\varepsilon)}{1-\exp(-\frac{\lambda}{C}\varepsilon)}\exp(\frac{\lambda}{C}(x-x_0))+\frac{C+\eta}{\lambda}-M,\\
\text{if}~v^\star(x_0)-v^\star(x_0-\varepsilon)\ge 0,\\
\ds \frac{v^\star(x_0)-v^\star(x_0-\varepsilon)}{1-\exp(\frac{\lambda}{C}\varepsilon)}\exp(-\frac{\lambda}{C}(x-x_0))+\frac{C+\eta}{\lambda}-M,\\
\text{if}~v^\star(x_0)-v^\star(x_0-\varepsilon)\leq 0,
\end{cases}
\end{eqnarray}
where the constant $C_M=C$ appears in assumption \eqref{hyp 1 order}.
Observe that $(\overline{\phi},\underline{\phi})$ both solve the following Eikonal equations
\begin{align}\label{eq : system EDO test ordre 1}
&\begin{cases}
\nonumber \lambda \overline{\phi}(x)-\lambda M+C\big(1+|\partial_x\overline{\phi}(x)|\big)=-\eta,~~x\in  (x_0-\varepsilon,x_0+\varepsilon),\\
\nonumber \overline{\phi}(x_0)-\overline{\phi}(x_0-\varepsilon)=u_\star(x_0)-u_\star(x_0-\varepsilon),    
\end{cases}\\
&\begin{cases}
\lambda \underline{\phi}(x)+\lambda M-C\big(1+|\partial_x\underline{\phi}(x)|\big)=\eta,~~x\in  (x_0-\varepsilon,x_0+\varepsilon),\\
\underline{\phi}(x_0)-\underline{\phi}(x_0-\varepsilon)=v^\star(x_0)-v^\star(x_0-\varepsilon).    
\end{cases}
\end{align}
Now, applying a slight quadratic variation to the last functions $(\overline{\phi},\underline{\phi})$, we introduce the two following maps $(f,g)$, defined for all $x\in [x_0-\varepsilon,x_0+\varepsilon]$ by
\begin{eqnarray}\label{expression f g}
f(x)=\overline{\phi}(x)-(x-x_0)^2/2,~~g(x)=\underline{\phi}(x)+(x-x_0)^2/2.\end{eqnarray}
These last functions will be used as test functions for the super solution $u$ (with $f$) and the sub solution $v$ (with $g$).
Now observe if we fix $\eta=\eta(\varepsilon,C)$ such that
\begin{eqnarray}\label{eq eta fixation}
    \eta_\varepsilon > C\varepsilon,
\end{eqnarray}
it follows that the super solution $u_\star$ and the test function $f$ satisfy as soon as the following equality $u_\star(x)=f(x)$ holds true, $\forall x \in (x_0-\varepsilon,x_0+\varepsilon)$ (with the aid of \eqref{born u v}, assumption $(\mathcal{H})$ and \eqref{eq : system EDO test ordre 1})
\begin{eqnarray}\label{eq stric H}
\nonumber H(x,u_\star(x),\partial_xf(x)) \leq \lambda (u_\star(x)-M)+C(1+|\partial_x\overline{f}(x)|)\leq\\
\nonumber \lambda \overline{\phi}(x)-\lambda (x-x_0)^2/2 -\lambda M+ C(1+|\partial_x\overline{\phi}(x)|)+C|x-x_0|\leq \\
-\eta_\varepsilon +C\varepsilon<0.
\end{eqnarray}
Using that $u$ is a super solution, it follows that if the minimum of $u_\star-f$ is reached at the interior of $(x_0-\varepsilon,x_0+\varepsilon)$ at $x$, one should have
$$H(x,u_\star(x),\partial_xf(x))\ge 0,~~u_\star(x)=f(x),$$
and then contradicting \eqref{eq stric H}.
Same arguments hold true for the maximum of $v^\star-g$ with the same choice of the parameter $\eta$ in \eqref{eq eta fixation}. 
Now observe also using the expressions given in \eqref{eq : system EDO test ordre 1} together with \eqref{expression f g}, we have
$$u_\star(x_0-\varepsilon)-f(x_0-\varepsilon)>u_\star(x_0)-f(x_0),$$
since
\begin{align*}
f(x_0)-f(x_0-\varepsilon)=\overline{\phi}(x_0)-\overline{\phi}(x_0-\varepsilon)+\varepsilon^2/2=u_\star(x_0)-u_\star(x_0-\varepsilon)+\varepsilon^2/2.
\end{align*}
Remark also
$$v^\star(x_0-\varepsilon)-g(x_0-\varepsilon)<v^\star(x_0)-g(x_0).$$
We are able to conclude that the minimum of $u_\star-f$ and the the maximum of $v^\star-g$ are both reached at $x_0+\varepsilon$. It follows that
\begin{eqnarray*}
&u_\star(x_0+\varepsilon)-f(x_0+\varepsilon)\leq u_\star(x_0)-f(x_0),
\\
&v^\star(x_0+\varepsilon)-g(x_0+\varepsilon)\ge v^\star(x_0)-g(x_0).\end{eqnarray*}
Using the last inequalities, \eqref{expression f g} and \eqref{eq contra ordre 1}, we get 
\begin{align}\label{eq hyp the base}
 &\nonumber-\varepsilon^2+[\overline{\phi}(x_0+\varepsilon)-\overline{\phi}(x_0)]-[\underline{\phi}(x_0+\varepsilon)-\underline{\phi}(x_0)]\\
&\ge v^\star(x_0)-u_\star(x_0) -[v^\star(x_0+\varepsilon)-u_\star(x_0+\varepsilon)]\ge 0. \end{align}
We are now in the right conditions to use Lemma \ref{lm of strong comparison theorem}, with the expression of $(\overline{\phi},\underline{\phi})$ given in \eqref{expression function test 1}-\eqref{expression function test 2}, 
(recall that $v^\star(x_0)-v^\star(x_0-\varepsilon)\ge u_\star(x_0)-u_\star(x_0-\varepsilon)$), to obtain
$$[\overline{\phi}(x_0+\varepsilon)-\overline{\phi}(x_0)]-[\underline{\phi}(x_0+\varepsilon)-\underline{\phi}(x_0)]\leq 0.$$
It follows that \eqref{eq hyp the base} leads to
$$-\varepsilon^2\ge 0,$$
and therefore a contradiction.

In conclusion of our analysis, the supremum of $v^\star-u_\star$ can not be reached at some $x_0\in(0,a)$, and is necessary reached at the boundaries. In other terms
$$\sup\big\{~v^\star(x)-u_\star(x),~x\in[0,a]~\big\}=\max\big(~v^\star(0)-u_\star(0),~v^\star(a)-u_\star(a)~\big).$$
The proof is complete.

\section{The central result}\label{sec preuve} 
In this section, we prove the main result of the paper, namely Theorem \ref{th : comparison}.

Since the proof is rather lengthy and involves several technical steps, we begin by outlining its main ideas.

\begin{itemize}

\item Let $u$ be a super solution and $v$ be a sub solution. Arguing by contradiction, we assume that the maximum of $v^\star-u_\star$ is attained at some interior point $x_0\in \Omega$.

\item Fix $\varepsilon>0$ such that $\mathcal{B}(x_0,\varepsilon)\subset \Omega$. We denote by $\overline{x}_\varepsilon$ and $\underline{x}_\varepsilon$ two diametrically opposite points on $\partial B(x_0,\varepsilon)$. Given $\delta>0$, we define $
\partial B_\delta(x_0,\varepsilon)
:=
\partial B(x_0,\varepsilon)\setminus B(\overline{x}_\varepsilon,\delta).$

\item Besides the parameters $\varepsilon$ and $\delta$, we introduce auxiliary parameters $\eta>0$, $C>0$, and $(\alpha,\beta)$. These parameters are chosen successively: first $C=C_\varepsilon$, then $(\alpha,\beta)$ as functions of $(\varepsilon,C_\varepsilon,\delta)$, and finally $\eta$ as a function of all the previously fixed parameters.

\item We first construct two auxiliary functions $(\overline{\phi},\underline{\phi})$ solving suitable Eikonal equations with appropriate boundary conditions. Their explicit expressions are tailored to the application of Lemma \ref{lm of strong comparison theorem}. They depend on $(\eta,C,\varepsilon)$, and the parameter $C=C_\varepsilon$ is selected so as to obtain a uniform control on their Hessians. These functions may be viewed as first-order building blocks for the test functions and are designed to make the Hamiltonian strictly negative or strictly positive throughout $B(x_0,\varepsilon)$.

\item We then introduce two test functions, denoted by $f$ and $g$, associated respectively with the super solution $u$ and the sub solution $v$. They are obtained by adding suitable quadratic perturbations, depending on $(\alpha,\beta)$, to $(\overline{\phi},\underline{\phi})$. The role of these perturbations is to prevent the extrema of $u_\star-f$ and $v^\star-g$ from being attained on $\partial B_\delta(x_0,\varepsilon)$, forcing them instead to lie near $\overline{x}_\varepsilon$. The parameter $\eta$ is finally chosen so that the Hamiltonian evaluated at $f$ and $g$ remains strictly negative and strictly positive, respectively, throughout $B(x_0,\varepsilon)$.
\item As a consequence, the extrema of $u_\star-f$ and $v^\star-g$ are attained in $
\partial B(x_0,\varepsilon)\cap B(\overline{x}_\varepsilon,\delta).$
We may then let $\delta\searrow 0$ and exploit the explicit expressions of $(\overline{\phi},\underline{\phi})$, together with Lemma \ref{lm of strong comparison theorem}, to derive the desired contradiction.
\end{itemize}
\textbf{Proof of the Strong Comparison Theorem \ref{th : comparison}.}
\begin{proof}
We proceed by contradiction and assume that the supremum of $v^\star-u_\star$ is attained at an interior point of $\overline{\Omega}$, namely
\begin{eqnarray}\label{eq contra 0}
\sup\big\{v^\star(x)-u_\star(x),~x\in \overline{\Omega}\}=v^\star(x_0)-u_\star(x_0),~~\text{for some}~~x_0\in \Omega
\end{eqnarray}
Let $\varepsilon>0$ be such that $\mathcal{B}(x_0,\varepsilon)\subset \Omega$. By assumption, both $u$ and $v$ are locally bounded. We denote by $M_\varepsilon$ a local bound for $u$ and $v$ on $\overline{\mathcal{B}(x_0,\varepsilon)}$, that is
$$\forall x\in \overline{\mathcal{B}(x_0,\varepsilon)},~~|u(x)|\leq M_\varepsilon,~~|v(x)|\leq M_\varepsilon.$$
We next introduce two points on the boundary $\partial \mathcal{B}(x_0,\varepsilon)$, denoted by $(\underline{x}_\varepsilon,\overline{x}_\varepsilon)$, whose coordinates are given by
\begin{eqnarray}\label{points de construction}
\forall i\in [\![1,N]\!],~~\underline{x}_{i,\varepsilon}=x_{i,0}-\varepsilon\mathbf{1}_{1,i},~~\overline{x}_{i,\varepsilon}=x_{i,0}+\varepsilon \mathbf{1}_{1,i}.   
\end{eqnarray}
Here, $\mathbf{1}_{i,j}$ denotes the Kronecker symbol, defined by $\mathbf{1}_{i,j}=1$ if $i=j$ and $\mathbf{1}_{i,j}=0$ otherwise.
Let $\delta\in(0,\pi\varepsilon)$ be another fixed parameter. We define the following subset $\partial B_{\delta}(x_0,\varepsilon)$ of the boundary $\partial \mathcal{B}(x_0,\varepsilon)$:
$$\partial B_{\delta}(x_0,\varepsilon)=\Big\{y\in \partial \mathcal{B}(x_0,\varepsilon),~~|\overline{x}_\varepsilon-y|_{_2}>\delta\Big\}.$$
\textbf{Step 1: Construction of the exponential barriers}\\
Let $\eta>0$ be a parameter that will be set last, after all the other parameters involved in this proof (see \eqref{fixation parametre eta}, Step 4). We now consider the following maps in the class $\mathcal{C}^2(\mathcal{B}(x_0,\varepsilon))$, denoted by $(\overline{\phi},\underline{\phi})=(\overline{\phi}^{\varepsilon,\delta,\eta},\underline{\phi}^{\varepsilon,\delta,\eta})$, and defined $\forall x \in \mathcal{B}(x_0,\varepsilon)$ by
\begin{eqnarray}\label{expression direct function test 1}
&\overline{\phi}(x)=
\begin{cases}\ds \frac{u_\star(x_0)-u_\star(\underline{x}_\varepsilon)}{1-\exp( \frac{\lambda}{NC}\varepsilon)}\exp(-\frac{\lambda}{NC}\sum_{i=1}^N(x_i-x_{i,0}))-\frac{C+\eta}{\lambda}+M_\varepsilon,\\
\text{if}~u_\star(x_0)-u_\star(\underline{x}_\varepsilon) \ge 0,\\
\ds \frac{u_\star(x_0)-u_\star(\underline{x}_\varepsilon)}{1-\exp(-\frac{\lambda}{NC}\varepsilon)}\exp(\ds \frac{\lambda}{NC}\sum_{i=1}^N(x_i-x_{i,0}))-\frac{C+\eta}{\lambda}+M_\varepsilon,\\
\text{if}~u_\star(x_0)-u_\star(\underline{x}_\varepsilon)\leq 0.
\end{cases}
\\
\label{expression direct function test 2}
&\underline{\phi}(x)=
\begin{cases}\ds \frac{v^\star(x_0)-v^\star(\underline{x}_\varepsilon)}{1-\exp(-\frac{\lambda}{NC}\varepsilon)}\exp(\frac{\lambda}{NC}\sum_{i=1}^N(x_i-x_{i,0}))+\frac{C+\eta}{\lambda}-M_\varepsilon,\\\text{if}~v^\star(x_0)-v^\star(\underline{x}_\varepsilon) \ge 0,\\
\ds \frac{v^\star(x_0)-v^\star(\underline{x}_\varepsilon)}{1-\exp(\frac{\lambda}{NC}\varepsilon)}\exp(\ds -\frac{\lambda}{NC}\sum_{i=1}^N(x_i-x_{i,0}))+\frac{C+\eta}{\lambda}-M_\varepsilon,\\
\text{if}~v^\star(x_0)-v^\star(\underline{x}_\varepsilon)\leq 0.
\end{cases}
\end{eqnarray}
Recall that $\lambda>0$ appears in assumption $(\mathcal{H}),~-a)$. Here $C>0$ is a constant that will play a central role in controlling uniformly the norm of the Hessian matrices of $(\overline{\phi},\underline{\phi})$ (see \eqref{def: constante Hessian}). Straightforward calculations show that both $\overline{\phi}$ and $\underline{\phi}$ satisfy the following eikonal equations
\begin{align}
&\begin{cases}\label{eq : system EDO 1}
\lambda \overline{\phi}(x)-\lambda M_\varepsilon+C\big(1+|\nabla\overline{\phi}(x)|_{_1}\big)=-\eta,~~x\in  \mathcal{B}(x_0,\varepsilon),\\ \overline{\phi}(x_0)-\overline{\phi}(\underline{x}_\varepsilon)=u_\star(x_0)-u_\star(\underline{x}_\varepsilon)   
\end{cases}
\\
&\begin{cases}\label{eq : system EDO 2}
\lambda \underline{\phi}(x)+\lambda M_\varepsilon-C\big(1+|\nabla\underline{\phi}(x)|_{_1}\big)=\eta,~~x\in  \mathcal{B}(x_0,\varepsilon),\\
\underline{\phi}(x_0)-\underline{\phi}(\underline{x}_\varepsilon)=v^\star(x_0)-v^\star(\underline{x}_\varepsilon)    
\end{cases}.
\end{align}
At this stage of the proof, we choose the parameter $C=C_\varepsilon>0$ large enough in order to uniformly control the norms of the Hessian matrices $(\nabla^2\overline{\phi},\nabla^2\underline{\phi})$. We begin with $\overline{\phi}$. We have that, for all $(i,j)\in \llbracket 1,N \rrbracket^2$ and all $x\in \mathcal{B}(x_0,\varepsilon)$
\begin{eqnarray*}
&\partial^2_{x_i,x_j}\overline{\phi}(x)=
\begin{cases}\ds \frac{u_\star(x_0)-u_\star(\underline{x}_\varepsilon)}{1-\exp( \frac{\lambda}{NC}\varepsilon)}\big(\frac{\lambda}{NC}\big)^2\exp(-\frac{\lambda}{NC}\sum_{i=1}^N(x_i-x_{i,0})),\\
\text{if}~u_\star(x_0)-u_\star(\underline{x}_\varepsilon) \ge 0,\\
\ds \frac{u_\star(x_0)-u_\star(\underline{x}_\varepsilon)}{1-\exp(-\frac{\lambda}{NC}\varepsilon)}\big(\frac{\lambda}{NC}\big)^2\exp(\ds \frac{\lambda}{NC}\sum_{i=1}^N(x_i-x_{i,0})),\\
\text{if}~u_\star(x_0)-u_\star(\underline{x}_\varepsilon)\leq 0.
\end{cases}
\end{eqnarray*}
Remarking that $\forall y>0$
$$\max(\frac{1}{\exp(y)-1},\frac{1}{1-\exp(-y)})\leq \frac{\exp(y)}{y},$$
and 
$$\sup\{\ds \frac{\lambda}{NC}\sum_{i=1}^N(x_i-x_{i,0}),~x\in \mathcal{B}(x_0,\varepsilon)\}=\exp(\frac{\lambda}{\sqrt{N}C}\varepsilon),$$
we obtain $\forall x\in \mathcal{B}(x_0,\varepsilon)$
\begin{eqnarray*}
&|\partial^2_{x_i,x_j}\overline{\phi}(x)|\ds\leq 
\frac{2M_\varepsilon}{\varepsilon}\frac{\lambda}{N C}\exp(\frac{\lambda}{NC}\varepsilon)\exp(\frac{\lambda}{\sqrt{N}C}\varepsilon).
\end{eqnarray*}
Therefore, there exists $\overline{C}_\varepsilon=C(\varepsilon,M_\varepsilon)$ large enough such that $\forall x\in \mathcal{B}(x_0,\varepsilon)$
\begin{eqnarray*}
|\partial^2_{x_i,x_j}\overline{\phi}(x)|\leq 1, ~~\forall C\ge C_\varepsilon,
\end{eqnarray*}
implying
\begin{eqnarray}\label{born Hessian phi sup}
\underset{x\in \mathcal{B}(x_0,\varepsilon)}{\sup}|\nabla^2\overline{\phi}(x))|_{_{\mathbb{M}_N(\R)}}\leq 1,~~\forall \varepsilon>0.
\end{eqnarray}
Same arguments, lead to
\begin{eqnarray*}
\underset{x\in \mathcal{B}(x_0,\varepsilon)}{\sup}|\nabla^2\underline{\phi}(x))|_{_{\mathbb{M}_N(\R)}}\leq 1,~~\forall \varepsilon>0.
\end{eqnarray*}
Finally recall that assumption $(\mathcal{H}),~-b)$ implies the existence of a constant $C_{M_\varepsilon,1}$ satisfying
\begin{align*}\forall (x,u,p,S)\in \mathcal{B}(x_0,\varepsilon)\times \R\times \R^N\times \mathbb{S}_N(\R),~~\text{if}~~|u|\leq M_\varepsilon,~~|S|\leq 1,\\
\text{then}~~|H(x,u,p,S)|\leq C_{M_\varepsilon,1}(1+|p|_{_{1}})^m,\end{align*}
for some $m\ge1 $.
We will set in the sequel of this proof, $C_\varepsilon$ s.t.
\begin{eqnarray}\label{def: constante Hessian}
C_\varepsilon=\max(\overline{C}_\varepsilon,C_{M_\varepsilon,1}),\end{eqnarray}
large enough in order to satisfy the following inequalities $\forall x\in \mathcal{B}(x_0,\varepsilon)$
\begin{eqnarray}\label{ineq clef Hamiltoniens}
&\nonumber|H(x,u(x),\nabla\overline{\phi}(x),\nabla^2\overline{\phi}(x))|\leq C_{M_\varepsilon,1}(1+|\nabla\overline{\phi}(x)|_{_{1}})^m\leq C_{\varepsilon}(1+|\nabla\overline{\phi}(x)|_{_{1}})^m.\\
&|H(x,u(x),\nabla\underline{\phi}(x),\nabla^2\underline{\phi}(x))|\leq C_{M_\varepsilon,1}(1+|\nabla\underline{\phi}(x)|_{_{1}})^m\leq C_{\varepsilon}(1+|\nabla\underline{\phi}(x)|_{_{1}})^m.
\end{eqnarray}
\textbf{Step 2: Quadratic penalization}\\
Let $(\alpha,\beta)=(\alpha^{\varepsilon,\delta},\beta^{\varepsilon,\delta})$ be two additional parameters, which will be fixed later in \eqref{cond alpha beta}. We begin by applying a quadratic perturbation to $(\underline{\phi},\overline{\phi})$, setting, for all $x\in \mathcal{B}(x_0,\varepsilon)$
\begin{align}\label{expr test function quadratisee}
\nonumber f^{\varepsilon,\delta,\eta}(x)=f(x)=\overline{\phi}(x)-\alpha \frac{|x-x_0|_{_2}^2}{2}-\beta \frac{|x-\overline{x}_\varepsilon|_{_2}^2}{2}\\ 
g^{\varepsilon,\delta,\eta}(x)=g(x)=\underline{\phi}(x)+\alpha \frac{|x-x_0|_{_2}^2}{2}+\beta \frac{|x-\overline{x}_\varepsilon|_{_2}^2}{2}
\end{align}
In the sequel of this proof, $f$ (resp. $g$) will be used as test functions for the super solution $u$ (resp. sub solution $v$). The main goal of this step is to choose the parameters $(\alpha,\beta)$ as functions of $(\varepsilon, C_\varepsilon, M_\varepsilon, \delta)$, independently of $\eta$, so that the minimum of $u_\star - f$ and the maximum of $v_\star - g$ cannot be attained on
\begin{align*}
\partial_\delta\mathcal{B}(x_0,\varepsilon)=\Big\{y\in \partial \mathcal{B}(x_0,\varepsilon),~~|\overline{x}_\varepsilon-y|_{_2}> \delta\Big\}.
\end{align*}
Let us investigate the case for the sub solution $u$ with $f$. Same arguments will be also available for the super solution $v$ and $g$.\\
For that purpose, let us impose that $\forall x \in \partial_\delta\mathcal{B}(x_0,\varepsilon)$
$$u_\star(x)-f(x)>u_\star(x_0)-f(x_0),$$
and then
$$f(x_0)-f(x)=(\alpha-\beta)\varepsilon^2/2+\beta \frac{|x-\overline{x}_\varepsilon|_{_2}^2}{2}+\overline{\phi}(x_0)-\overline{\phi}(x)>u_\star(x_0)-u_\star(x).$$
Regarding the expression of $\overline{\phi}$ given in \eqref{expression direct function test 1}, and using the fact that the variation $\overline{\phi}(x_0)-\overline{\phi}(x)$ is independent of $\eta$, it follows that the last inequality holds provided that we impose
\begin{align}\label{cond alpha beta}
\nonumber &\alpha>\beta,\\
&\beta(\varepsilon,C_\varepsilon,M_\varepsilon,\delta)=\beta >\frac{2}{\delta^2}\big[2M_\varepsilon+\frac{2M_\varepsilon NC_\varepsilon}{\varepsilon \lambda}\exp(\frac{\lambda}{2\sqrt{N}C_\varepsilon}\varepsilon)\big],
\end{align}
using the same upper bounds used for the Hessian matrix as in the previous Step 2.\\
\textbf{Step 3. Exclusion of interior contact points}\\
In this step, we finally choose the last parameter $\eta=\eta(\varepsilon,\delta,\alpha^{\varepsilon,\delta},\beta^{\varepsilon,\delta},C_\varepsilon,M_\varepsilon)$ (see \eqref{fixation parametre eta}) in order to force the Hamiltonian evaluated at the test function $f$ (resp. $g$) to be strictly negative (resp. strictly positive) at any point of $\mathcal{B}(x_0,\varepsilon)$. Recall that this will imply that the minimum of $u_\star-f$ (resp. the maximum of $v^\star-g$) cannot be attained at any point in $\mathcal{B}(x_0,\varepsilon)$.

We first note, using \eqref{expr test function quadratisee}, that we have
\begin{align*}
|\nabla f(x)|_{_{1}}\leq |\nabla\overline{\phi}(x)|_{_{1}}+\varepsilon \sqrt{N}(\alpha + \beta),\\
\nabla^2f(x)=\nabla^2\overline{\phi}(x)-(\alpha+\beta)I_N,~~\forall x\in \mathcal{B}(x_0,\varepsilon). 
\end{align*}
Using \eqref{born Hessian phi sup}, it follows that
\begin{align*}|\nabla^2f(x)|_{_{\mathbb{M}_N(\R)}}\leq 1+(\alpha+\beta).
\end{align*}
We now observe that, for any $y\in \mathcal{B}(x_0,\varepsilon)$, whenever $u_\star(y)=f(y)$, and using assumption $(\mathcal{H})$, together with the bound $u_\star(y)\leq M_\varepsilon$ and the eikonal equation \eqref{eq : system EDO 1}
\begin{align}\label{sign HJB}
&\nonumber H(y,u_\star(y),\nabla f(y),\nabla^2f(y))\leq \lambda \overline{\phi}(y)-\lambda M_\varepsilon-\lambda\Bigg[\alpha \frac{\left|y-x_0\right|_{_2}^2}{2}+\beta \frac{\left|y-\overline{x}_\varepsilon\right|_{_2}^2}{2}\Bigg]\\
&\nonumber +C_{M_\varepsilon,1+(\alpha+\beta)}\big(1+|\nabla \overline{\phi}(y)|+(\alpha+\beta)\varepsilon\sqrt{N}\big)^m\leq \\
&\nonumber -\eta-C_\varepsilon\big(1+|\nabla\overline{\phi}(y)|\big)+C_{M_\varepsilon,1+(\alpha+\beta)}\big(1+|\nabla \overline{\phi}(y)|+(\alpha+\beta)\varepsilon\sqrt{N}\big)^m.
\end{align}
The key point in choosing the parameter $\eta>0$ so as to force the last expression to be strictly negative is to use the fact that the gradient of $\overline{\phi}$ is independent of $\eta$. Hence, as soon as
$$\eta>C_{M_\varepsilon,1+(\alpha+\beta)}\big(1+|\nabla \overline{\phi}(y)|+(\alpha+\beta)\varepsilon\sqrt{N}\big)^m,~~\forall y \in \mathcal{B}(x_0,\varepsilon),$$
for instance
\begin{eqnarray}\label{fixation parametre eta}
\eta>C_{M_\varepsilon,1+(\alpha+\beta)}\big(1+\frac{2NM_\varepsilon}{\varepsilon}\exp(\frac{\lambda}{2\sqrt{N}C_\varepsilon}\varepsilon)+(\alpha+\beta)\varepsilon\sqrt{N}\big)^m,
\end{eqnarray}
we obtain
\begin{eqnarray}\label{ineq HJB}
H(y,u(y),\nabla f(y),\nabla^2f(y))<0.  
\end{eqnarray}
Recall also that $(\alpha,\beta)$ were fixed in \eqref{cond alpha beta}, independently of $\eta$. The same arguments apply to the sub solution $v$. The goal of this step is thus achieved.\\
\textbf{Step 4. Localization of the contact points}\\
Combining Steps 2 and 3 with the choice of parameters in the construction of the test function $f$, and using that $u$ is a super solution, the previous arguments allow us to conclude that the minimum of $u_\star - f$ cannot be attained at any point in the open ball $\mathcal{B}(x_0,\varepsilon)$, nor at any point of $\partial \mathcal{B}_\delta(x_0,\varepsilon)$.\\
We can therefore assert that there exists $y_\varepsilon^\delta \in \partial\mathcal{B}(x_0,\varepsilon)\setminus \partial \mathcal{B}_\delta(x_0,\varepsilon)$ such that
$$\min\big\{u_\star(y)-f(y),~y\in \overline{\mathcal{B}(x_0,\varepsilon)}\big\}=u_\star(y_\varepsilon^\delta)-f(y_\varepsilon^\delta).$$
The same arguments apply to the sub solution $v$, with the same choice of parameters $(\alpha,\beta,\eta)$ and using $\underline{\phi}$, yielding the existence of $z_\varepsilon^\delta \in \partial\mathcal{B}(x_0,\varepsilon)\setminus \partial \mathcal{B}_\delta(x_0,\varepsilon)$ such that
$$\max\big\{v^\star(y)-g(y),~y\in \overline{\mathcal{B}(x_0,\varepsilon)}\big\}=v^\star(z_\varepsilon^\delta)-g(z_\varepsilon^\delta).$$
We obtain then
$$u_\star(x_0)-f(x_0)\ge u(y_\varepsilon^\delta)-f(y_\varepsilon^\delta),~~v^\star(x_0)-g(x_0)\leq v(z_\varepsilon^\delta)-g(z_\varepsilon^\delta),$$
and therefore
\begin{eqnarray}\label{eq hyp the base 2 order}
[f(y_\varepsilon^\delta)-f(x_0)]-[g(z_\varepsilon^\delta)-g(x_0)]\ge 
v^\star(x_0)-u_\star(x_0) -[v(z_\varepsilon^\delta)-u(y_\varepsilon^\delta)].   
\end{eqnarray}
On the other hand as $\delta\searrow 0$, we have
\begin{eqnarray}\label{lim des points}
\lim_{\delta\searrow 0}z_\varepsilon^\delta=\overline{x}_\varepsilon,~~\lim_{\delta\searrow 0}y_\varepsilon^\delta=\overline{x}_\varepsilon,  \end{eqnarray}
(recall $\partial\mathcal{B}(x_0,\varepsilon)\setminus \partial \mathcal{B}_\delta(x_0,\varepsilon)=\Big\{~y\in \partial \mathcal{B}(x_0,\varepsilon),~~|\overline{x}_\varepsilon-y|_{_2}\leq \delta~\Big\}$.)\\
\\
\textbf{Step 5. Contradiction}\\
From \eqref{expr test function quadratisee}, it follows that
\begin{eqnarray}\label{eq inegal}
\nonumber&[f(y_\varepsilon^\delta)-f(x_0)]-[g(z_\varepsilon^\delta)-g(x_0)]=[\overline{\phi}(y_\varepsilon^\delta)-\overline{\phi}(x_0)]-[\underline{\phi}(z_\varepsilon^\delta)-\underline{\phi}(x_0)]\\
&\nonumber-\ds(\alpha-\beta) \varepsilon^2-\beta\frac{|y_\varepsilon^\delta-\overline{x}_\varepsilon|_{_2}^2}{2}-\beta\frac{|z_\varepsilon^\delta-\overline{x}_\varepsilon|_{_2}^2}{2}\\
&\leq [\overline{\phi}(y_\varepsilon^\delta)-\overline{\phi}(x_0)]-[\underline{\phi}(z_\varepsilon^\delta)-\underline{\phi}(x_0)]-\varepsilon^2,
\end{eqnarray}
if, for instance, $\alpha=\beta+1$, which is possible and consistent with the conditions \eqref{cond alpha beta}. Now, inspecting the expressions of the variations $\overline{\phi}(y_\varepsilon^\delta)-\overline{\phi}(x_0)$ and $\underline{\phi}(z_\varepsilon^\delta)-\underline{\phi}(x_0)$ given in \eqref{expression direct function test 1}-\eqref{expression direct function test 2}, it follows that, by letting $\delta \searrow 0$ and using \eqref{lim des points}
\begin{align*}
&\underset{\delta\searrow 0}{\lim}~\overline{\phi}(y_\varepsilon^\delta)-\overline{\phi}(x_0)=
\begin{cases}\ds \frac{u_\star(x_0)-u_\star(\underline{x}_\varepsilon)}{1-\exp( \frac{\lambda}{NC_\varepsilon}\varepsilon)}\big[\exp(-\frac{\lambda}{NC_\varepsilon}\varepsilon)-1\big],~&\text{if}~u_\star(x_0)-u_\star(\underline{x}_\varepsilon) \ge 0,\\
\ds \frac{u_\star(x_0)-u_\star(\underline{x}_\varepsilon)}{1-\exp(-\frac{\lambda}{NC_\varepsilon}\varepsilon)}\big[\exp(\ds \frac{\lambda}{NC_\varepsilon}\varepsilon)-1\big],~&\text{if}~u_\star(x_0)-u_\star(\underline{x}_\varepsilon)\leq 0.
\end{cases}
\\
&\underset{\delta\searrow 0}{\lim}~\underline{\phi}(y_\varepsilon^\delta)-\underline{\phi}(x_0)=
\begin{cases}\ds \frac{v^\star(x_0)-v^\star(\underline{x}_\varepsilon)}{1-\exp(-\frac{\lambda}{NC_\varepsilon}\varepsilon)}\big[\exp(\ds \frac{\lambda}{NC_\varepsilon}\varepsilon)-1\big],~&\text{if}~v^\star(x_0)-v^\star(\underline{x}_\varepsilon) \ge 0,\\
\ds \frac{v^\star(x_0)-v^\star(\underline{x}_\varepsilon)}{1-\exp(\frac{\lambda}{NC_\varepsilon}\varepsilon)}\big[\exp(-\frac{\lambda}{NC_\varepsilon}\varepsilon)-1\big],~&\text{if}~v^\star(x_0)-v^\star(\underline{x}_\varepsilon)\leq 0.
\end{cases}
\end{align*}
Recall that the assumption \eqref{eq contra 0} at the beginning of the proof, implies that
$$v^\star(x_0)-u_\star(x_0) \ge v^\star(\underline{x}_\varepsilon)-u_\star(\underline{x}_\varepsilon).$$
We can conclude from Lemma \ref{lm of strong comparison theorem}
$$\underset{\delta\searrow 0}{\lim}~\overline{\phi}(y_\varepsilon^\delta)-\overline{\phi}(x_0)\leq \underset{\delta\searrow 0}{\lim}~\underline{\phi}(y_\varepsilon^\delta)-\underline{\phi}(x_0).$$
Hence \eqref{eq inegal} leads to
$$\underset{\delta\searrow 0}{\overline{\lim}}~[f(y_\varepsilon^\delta)-f(x_0)]-[g(z_\varepsilon^\delta)-g(x_0)]\leq -\varepsilon^2.$$
Finally the contradiction is obtained using both expressions \eqref{eq hyp the base 2 order} and \eqref{lim des points} that imply
$$\underset{\delta\searrow 0}{\underline{\lim}}~[f(y_\varepsilon^\delta)-f(x_0)]-[g(z_\varepsilon^\delta)-g(x_0)]~\ge~\underset{\delta\searrow 0}{\underline{\lim}}~v^\star(x_0)-u_\star(x_0) -[v^\star(z_\varepsilon^\delta)-u_\star(y_\varepsilon^\delta)]$$
$$=v^\star(x_0)-u_\star(x_0) -[v^\star(\overline{x}_\varepsilon)-u_\star(\overline{x}_\varepsilon)]\ge 0.$$
The proof is complete. 
\end{proof}
\section{Well-posedness}\label{sec existence} 
In this final section, we establish the well-posedness of \eqref{eq: EDP 1}, namely the existence and uniqueness of continuous viscosity solutions, for discontinuous Hamiltonians through the use of Ishii's upper and lower semicontinuous envelopes in the definition of viscosity solutions. We conclude with a remark outlining how the comparison principle established in Theorem \ref{th : comparison} can be extended to the parabolic setting.

Consider the nonlinear Dirichlet problem \eqref{eq: EDP 1}
\begin{eqnarray*}
\begin{cases}
H\big(x,u(x),\nabla u(x),\nabla^2u(x)\big)=0,~~x\in \Omega,\\
u(x)=G(x),~~x\in \partial \Omega,
\end{cases}
\end{eqnarray*}
where $H$ satisfies assumption $(\mathcal{H})$.

Due to the growth condition $(\mathcal{H}),~-b)$, the upper and lower semicontinuous envelopes of $H$, denoted by $H^\star$ and $H_\star$, are well defined on $\Omega\times \mathbb{R}\times \mathbb{R}^N\times \mathbb{S}_N(\mathbb{R})$ by
\begin{eqnarray*} 
H^\star(x,u,p,S)=\underset{(y,v,q,T)\to (x,u,p,S)}{\overline{\lim}} H(y,v,q,T) \end{eqnarray*} \begin{eqnarray*} 
H_\star(x,u,p,S)=\underset{(y,v,q,T)\to (x,u,p,S)}{\underline{\lim}} H(y,v,q,T). \end{eqnarray*}
Standard arguments based on the definition of semicontinuous envelopes and subsequence extraction show that both $H^\star$ and $H_\star$ inherit assumption $(\mathcal{H})$, up to a modification of the constant $C_{M,K}$ in $(\mathcal{H}),~-b)$. Consequently, when defining viscosity sub- and super solutions in the sense of Definition \ref{def: solu de visco}, replacing $H$ by $H^\star$ in the super solution case and by $H_\star$ in the sub solution case, the arguments in the proof of Theorem \ref{th : comparison} remain valid. In particular, the strong comparison principle continues to hold in the framework of Ishii's semicontinuous envelopes. We refer to this notion as Ishii viscosity solutions, where $H^\star$ is used for super solutions and $H_\star$ for sub solutions.

For the existence, the main argument of Perron’s method still holds true. Assume first that the boundary condition appearing in \eqref{eq: EDP 1} is bounded, namely
$$\forall x \in \partial \Omega,~~|G(x)|\leq G_\infty.$$
Using assumption $(\mathcal{H})$, it is easy to check that both constant maps $x\mapsto \max(G_\infty,\frac{C_{0.0}}{\lambda})$ and $x\mapsto \min(-G_\infty,-\frac{C_{0,0}}{\lambda})$ are respectively super and sub solution. (Recall that $\lambda>0$ appears in $(\mathcal{H}),~-a)$, whereas $C_{0,0}$ is related to $(\mathcal{H}),~-b)$, with $\forall x\in \Omega, ~|H(x,0,0,0)|\leq C_{0,0}$.)

As explained in Remark 4.5 of \cite{CrandallIshiiLions1992}, the existence of possibly discontinuous viscosity solutions is a direct consequence of Perron’s method. Since $H_\star$ is lower semicontinuous and $H^\star$ is upper semicontinuous, Perron’s method can be adapted to the framework of Ishii’s envelopes. The stability of the viscosity inequalities is then obtained via the sequence characterization of the envelopes.

To obtain the existence of a continuous viscosity solution, one needs a comparison principle (which is available in view of the preceding discussion and Theorem \ref{th : comparison}), together with the existence of sub- and super solutions satisfying the Dirichlet boundary condition.\\
We emphasize the strength of the method developed in the comparison theorem of this work. In classical approaches, one is required to control the difference between $H_\star$ and $H^\star$, which may be arbitrarily large in the discontinuous setting. Here, we bypass this technical difficulty through a purely local argument, relying instead on the construction of suitable test functions for $H_\star$ and $H^\star$. Theses test functions implicitly encode the values of the sub- and super solutions and lead to the desired comparison result.\\
In the same spirit then of Theorem 4.1 of \cite{CrandallIshiiLions1992}, we can then state the following well-posedness Theorem.
\begin{Theorem}\label{th well posedness}(Well-posedness).
Assume assumption $(\mathcal{H})$ and that the Dirichlet boundary condition $G \in \mathcal{C}(\partial \Omega,\mathbb{R})$ is continuous. Assume moreover that there exist a super solution $u$ and a sub solution $v$ of \eqref{eq: EDP 1} such that
$$
\forall x \in \partial \Omega,\quad u(x)=v(x)=G(x).
$$
Then there exists a unique continuous viscosity solution in the sense of Ishii of \eqref{eq: EDP 1} satisfying the boundary condition $G$.
\end{Theorem}

It is well known that the existence of such sub- and super solutions is a key point in the above statement, since compatibility conditions between $H$ and $G$ are required. For instance, in the model case $u=h(x)$ in $\Omega$ and $u=0$ on $\partial \Omega$, no continuous solution can generally be expected unless appropriate compatibility conditions are satisfied, even when $h$ is continuous.

Our result nevertheless provides a framework ensuring the existence of a unique continuous solution, when the Hamiltonian $H$ is merely measurable, as soon as such boundary compatibility holds. This is not standard in the literature.

In the continuous setting, we also refer to the discussion at the end of Section 7 of \cite{CrandallIshiiLions1992} and the references therein for related existence results for Dirichlet problems.

\begin{Remark}(The parabolic framework). In order not to repeat the same arguments used in this contribution for the parabolic framework, which would make the presentation heavy and repetitive, we prefer through this remark to give the reader the broad outline of the extension of our main result to this framework. It would also be more interesting to give a rigorous and precise proof in a new contribution, where the link for example with stochastic control, differential games, or martingale problems would be studied, in the framework of discontinuous borelian coefficients.
In this context, let us briefly outline the main ideas. In the parabolic setting, one may assume that conditions $(\mathcal{H}),~-b),~-c)$ hold uniformly with respect to time for a given Hamiltonian, which may itself depend on time and be discontinuous in all its variables. On the other hand, condition $(\mathcal{H}),~-a)$ generally fails in the parabolic framework. Nevertheless, it can be recovered through a standard perturbation argument consisting in multiplying the sub solution and the super solution by the factor $\exp(-\lambda t),~\lambda>0$, for $t\in[0,T]$. This transformation slightly modifies the Hamiltonian through exponential scalings, while preserving the structural assumptions embodied in $(\mathcal{H})$, and retrieves $(\mathcal{H}),~-a)$, namely the strict monotonicity of the Hamiltonian with respect to the unknown.

The key point is then to employ the same first-order test jets $(\overline{\phi},\underline{\phi})$ as those introduced in the proof of Theorem \ref{th : comparison}, which depend only on the spatial variable. Retaining their explicit expressions are crucial for applying Lemma \ref{lm of strong comparison theorem}, obtaining uniform control of the associated Hessian matrices, and exploiting the choice of the parameter $\eta$ to enforce a constant sign of the Hamiltonian.

Next, in addition to the quadratic penalization in the spatial variables, one introduces a quadratic penalization in time. This will ensure that the supremum point of the difference between test function and viscosity solution, is attained on a suitable point $(\overline{x}_\varepsilon,\overline{t}_\varepsilon)$, belonging to the boundary of the neighborhood where the supremum of the difference between the sub solution and the super solution is achieved. The argument can then be completed exactly as in the proof of Theorem \ref{th : comparison}.
\end{Remark}


\begin{thebibliography}{99} 

\bibitem{Achdou 1}
Y. Achdou, M.-K. Dao, O. Ley, and N. Tchou,
A class of infinite horizon mean field games on networks,
Netw. Heterog. Media 14 (2019), no. 3, 537--566.

\bibitem{Achdou 2}
Y. Achdou, M.-K. Dao, O. Ley, and N. Tchou,
Finite horizon mean field games on networks,
Calc. Var. Partial Differential Equations 59 (2020), no. 5, Paper No. 157, 34 pp.

\bibitem{Aleksandrov1958}
A. D. Aleksandrov,
Maximum principles for elliptic equations,
Vestnik Leningrad Univ. 13 (1958), 5--24.

\bibitem{Bakelman1966}
I. J. Bakelman,
Convex analysis and nonlinear elliptic equations,
Uspekhi Mat. Nauk 21 (1966), 3--78.

\bibitem{control 1}
G. Barles, A. Briani, and E. Chasseigne,
A Bellman approach for two-domains optimal control problems in $\mathbb{R}^N$,
ESAIM Control Optim. Calc. Var. 19 (2013), no. 3, 710--739.

\bibitem{control 2}
G. Barles, A. Briani, and E. Chasseigne,
A Bellman approach for regional optimal control problems in $\mathbb{R}^N$,
SIAM J. Control Optim. 52 (2014), no. 3, 1712--1744.

\bibitem{control 3}
G. Barles and E. Chasseigne,
(Almost) everything you always wanted to know about deterministic control problems in stratified domains,
Netw. Heterog. Media 10 (2015), no. 4, 809--836.

\bibitem{Barles semi linear}
G. Barles, O. Ley, and E. Topp,
Degenerate elliptic PDEs on a network with Kirchhoff conditions,
preprint, arXiv:2509.12848.

\bibitem{Barlesbook}
G. Barles and E. Chasseigne,
On Modern Approaches of Hamilton--Jacobi Equations and Control Problems with Discontinuities:
A Guide to Theory, Applications, and Some Open Problems,
Progress in Nonlinear Differential Equations and Their Applications, vol. 104,
Birkhäuser, 2024.

\bibitem{BassPardoux}
R. F. Bass and \'E. Pardoux,
Uniqueness for diffusions with piecewise constant coefficients,
Probab. Theory Related Fields 76 (1987), 557--572.

\bibitem{control 4}
A. Bressan and Y. Hong,
Optimal control problems on stratified domains,
Netw. Heterog. Media 2 (2007), no. 2, 313--331.

\bibitem{Kumar}
V. S. Borkar and K. Suresh Kumar,
A new Markov selection procedure for degenerate diffusions, Journal of Theoretical Probability, Volume 23, (2010), pages 729–747.

\bibitem{CaffarelliCrandallKocanSwiech1996}
L. A. Caffarelli, M. G. Crandall, M. Kocan, and A. \'Swiech,
On viscosity solutions of fully nonlinear equations with measurable ingredients,
Comm. Pure Appl. Math. 49 (1996), 365--397.

\bibitem{Caffarelli1989Annals}
L. A. Caffarelli,
Interior a priori estimates for solutions of fully nonlinear equations,
Ann. of Math. (2) 130 (1989), 189--213.

\bibitem{Caffarelli1989}
L. A. Caffarelli,
Interior a priori estimates for solutions of fully nonlinear equations,
Comm. Pure Appl. Math. 42 (1989), no. 3, 271--289.

\bibitem{CalderonZygmund1952}
A. P. Calder\'on and A. Zygmund,
On singular integrals,
Acta Math. 88 (1952), 85--139.

\bibitem{Camilli 1}
F. Camilli and C. Marchi,
Stationary mean field games systems defined on networks,
SIAM J. Control Optim. 54 (2016), no. 2, 1085--1103.

\bibitem{Camilli 2}
F. Camilli, C. Marchi, and D. Schieborn,
The vanishing viscosity limit for Hamilton--Jacobi equations on networks,
J. Differential Equations 254 (2013), no. 10, 4122--4143.

\bibitem{Carda junction -1}
P. Cardaliaguet,
A note on contractive semi-groups on a 1:1 junction for scalar conservation laws and Hamilton--Jacobi equations,
preprint, (2024) arXiv:2411.12326.

\bibitem{Carda junction 0}
P. Cardaliaguet,
A microscopic derivation of a traffic flow model on a junction with two entry lines,
ESAIM Math. Model. Numer. Anal. 59 (2025), 2021--2054.

\bibitem{Carda junction 1}
P. Cardaliaguet, N. Forcadel, T. Girard, and R. Monneau,
Conservation laws and Hamilton--Jacobi equations on a junction: the convex case,
Discrete Contin. Dyn. Syst. 44 (2024), no. 12, 3920--3961.

\bibitem{Carda junction 2}
P. Cardaliaguet and P. E. Souganidis,
An optimal control problem of traffic flow on a junction,
ESAIM Control Optim. Calc. Var. 30 (2024), Paper No. 88.

\bibitem{CeruttiEscauriazaFabes1}
M. C. Cerutti, L. Escauriaza, and E. Fabes,
Uniqueness for some elliptic problems with discontinuous coefficients,
Ann. Scuola Norm. Sup. Pisa Cl. Sci. (4) 25 (1997), 537--550.

\bibitem{CeruttiEscauriazaFabes2}
M. C. Cerutti, L. Escauriaza, and E. Fabes,
Green's function estimates and uniqueness for elliptic equations,
Ann. Inst. H. Poincar\'e Anal. Non Lin\'eaire 14 (1997), 1--29.

\bibitem{chiarenza1991}
F. Chiarenza, M. Frasca, and P. Longo,
$L^p$ estimates for second order elliptic equations with non-smooth coefficients,
Rend. Mat. Appl. (7) 10 (1990), 27--38.

\bibitem{chiarenza1993}
F. Chiarenza, M. Frasca, and P. Longo,
$W^{2,p}$-solvability of elliptic equations with VMO coefficients,
Trans. Amer. Math. Soc. 336 (1993), 841--853.

\bibitem{ChernobaiShilkin2024}
M. Chernobai and T. Shilkin,
Elliptic equations with a singular drift from a weak Morrey class,
preprint (2024).

\bibitem{CrandallIshiiLions1992}
M. G. Crandall, H. Ishii, and P.-L. Lions,
User's guide to viscosity solutions,
Bull. Amer. Math. Soc. (N.S.) 27 (1992), 1--67.

\bibitem{DongKrylov2010VMO}
H. Dong and N. V. Krylov,
Second-order elliptic and parabolic equations with coefficients measurable in one variable and VMO in others,
Trans. Amer. Math. Soc. 362 (2010), 6477--6494.

\bibitem{EscauriazaBounds}
L. Escauriaza,
Bounds for the fundamental solution of elliptic and parabolic equations,
Comm. Partial Differential Equations 25 (2000), no. 5--6, 821--845.

\bibitem{FabesStroock}
E. B. Fabes and D. W. Stroock,
The $L^p$-integrability of Green's functions and fundamental solutions,
Duke Math. J. 51 (1984), 997--1016.

\bibitem{FabesKenig}
E. B. Fabes and C. E. Kenig,
Examples of singular parabolic equations,
Indiana Univ. Math. J. 33 (1984), 559--571.

\bibitem{GilbargTrudinger1983}
D. Gilbarg and N. S. Trudinger,
Elliptic Partial Differential Equations of Second Order,
Springer, 1983.

\bibitem{control 5}
C. Imbert and R. Monneau,
Flux-limited solutions for quasi-convex Hamilton--Jacobi equations on networks,
Ann. Sci. Éc. Norm. Supér. (4) 50 (2017), no. 2, 357--448.

\bibitem{control 6}
C. Imbert and V. D. Nguyen,
Generalized junction conditions for degenerate parabolic equations,
Calc. Var. Partial Differential Equations 56 (2017), no. 3.

\bibitem{Ishii1985}
H. Ishii,
Hamilton--Jacobi equations with discontinuous Hamiltonians,
Bull. Fac. Sci. Eng. Chuo Univ. 28 (1985), 33--52.

\bibitem{Jensen1988}
R. Jensen,
Uniqueness of viscosity solutions of Hamilton--Jacobi equations,
Trans. Amer. Math. Soc. 318 (1990), 655--670.

\bibitem{Krylov1973}
N. V. Krylov,
Selection theorem for diffusion processes,
Izv. Akad. Nauk SSSR Ser. Mat. 37 (1973), 691--708.

\bibitem{Krylov1987}
N. V. Krylov,
Nonlinear Elliptic and Parabolic Equations of the Second Order,
D. Reidel, 1987.

\bibitem{Krylov1996lectures}
N. V. Krylov,
Lectures on Elliptic and Parabolic Equations in Sobolev Spaces,
American Mathematical Society, Providence, RI, 1996.

\bibitem{KrylovWeakUniq}
N. V. Krylov,
On weak uniqueness for some diffusions with discontinuous coefficients,
Stochastic Processes Appl. 113 (2004), 37--64.

\bibitem{KrylovSafonov1980}
N. V. Krylov and M. V. Safonov,
A property of the solutions of parabolic equations with measurable coefficients,
Izv. Akad. Nauk SSSR Ser. Mat. 44 (1980), 161--175.

\bibitem{KrylovSafonov1981}
N. V. Krylov and M. V. Safonov,
Harnack inequality for parabolic equations,
Math. USSR Izv. 16 (1981), 151--164.


\bibitem{KrylovVMO2007}
N. V. Krylov,
Parabolic and elliptic equations with VMO coefficients,
Commun. Partial Differential Equations 32 (2007), 453--475.

\bibitem{Krylov2023Morrey}
N. V. Krylov,
Elliptic equations in Sobolev spaces with Morrey drift and zeroth-order coefficients,
preprint (2023).

\bibitem{Lady1968}
O. A. Ladyzhenskaya and N. N. Ural'tseva,
Linear and Quasilinear Equations of Parabolic Type,
Transl. Math. Monogr. 23, American Mathematical Society, Providence, RI, 1968.

\bibitem{Lee2025Div}
H. Lee,
Well-posedness of linear elliptic equations with $L^d$ drifts under divergence-type conditions,
preprint (2025).

\bibitem{Lee2025General}
H. Lee,
Analysis of linear elliptic equations with general drifts and $L^1$ zeroth-order terms,
preprint (2025).


\bibitem{LionsPerthame1987}
P.-L. Lions and B. Perthame,
Hamilton--Jacobi equations with measurable Hamiltonians,
Nonlinear Anal. 11 (1987), 613--621.

\bibitem{Lions Souganidis 1}
P.-L. Lions and P. E. Souganidis,
Viscosity solutions for junctions: well-posedness and stability,
Rend. Lincei Mat. Appl. 27 (2016), 519--545.

\bibitem{Lions Souganidis 2}
P.-L. Lions and P. E. Souganidis,
Well-posedness for multi-dimensional junction problems with Kirchhoff-type conditions,
Rend. Lincei Mat. Appl. 28 (2017), 465--480.

\bibitem{Maugeri1999}
A. Maugeri, F. Petralia, and others,
Elliptic and Parabolic Equations with Discontinuous Coefficients,
Wiley, 1999.

\bibitem{Nadirashvili1997}
N. Nadirashvili,
Nonuniqueness for elliptic equations,
Ann. Scuola Norm. Sup. Pisa Cl. Sci. (4) 24 (1997), 537--550.

\bibitem{Nunziante1990}
D. Nunziante,
Existence and uniqueness of viscosity solutions of parabolic equations with discontinuous time dependence,
Nonlinear Anal. 14 (1990), 259--273.

\bibitem{Nunziante1991}
D. Nunziante,
Viscosity solutions of fully nonlinear parabolic equations with discontinuous coefficients,
J. Differential Equations 93 (1991), 226--244.

\bibitem{Ohavi visco 1}
I. Ohavi,
Comparison principle for Walsh's spider HJB equations with nonlinear local time Kirchhoff boundary transmission,
J. Math. Anal. Appl. 547 (2025), no. 2.

\bibitem{Ohavi visco 2}
I. Ohavi,
Viscosity solutions posed on star-shaped networks with Kirchhoff boundary conditions: well-posedness,
preprint, in progress in Advances in Nonlinear Analysis, (2026), arXiv:2510.14364.

\bibitem{OhaviPDE}
I. Ohavi,
Quasilinear parabolic PDEs posed on a network with nonlinear Neumann boundary conditions at vertices,
J. Math. Anal. Appl. 500 (2021), no. 1.

\bibitem{Ohavi control}
I. Ohavi,
Stochastic scattering control of spider diffusion governed by an optimal probability diffraction measure selected from its own local time, to appear in SIAM Journal of Control and Optimization (2026).

\bibitem{Pucci1966}
C. Pucci,
Maximum and minimum eigenvalues,
Ann. Mat. Pura Appl. 72 (1966), 141--170.

\bibitem{Safonov1980}
M. V. Safonov,
Harnack inequality for elliptic equations,
J. Soviet Math. 21 (1983), 851--863.

\bibitem{SafonovNonuniq}
M. V. Safonov,
Nonuniqueness for elliptic equations,
SIAM J. Math. Anal. 30 (1999), 879--895.

\bibitem{Sarason1975}
D. Sarason,
Functions of vanishing mean oscillation,
Trans. Amer. Math. Soc. 207 (1975), 391--405.

\bibitem{StroockVaradhan1979}
D. W. Stroock and S. R. S. Varadhan,
Multidimensional Diffusion Processes,
Springer, 1979.

\bibitem{Below 3}
J. von Below,
A maximum principle for semilinear parabolic network equations,
in Differential Equations with Applications in Biology, Physics, and Engineering (Leibniz, 1989),
Lecture Notes in Pure and Appl. Math. 133, Dekker, New York, 1991.

\end{thebibliography}
\end{document}